\documentclass[12pt, a4paper]{article}

\usepackage{amsmath}
\usepackage{amsopn}
\usepackage{amsfonts}
\usepackage{amsthm}
\usepackage{epsfig, psfrag}

\numberwithin{equation}{section}

\theoremstyle{remark}
\newtheorem{remark}{Remark}
\newtheoremstyle{mytheorem}{0.5cm}{0.2cm}{\slshape}{ }{\bfseries}{.}{ }{}
\theoremstyle{mytheorem}
\newtheorem{theorem}{Theorem}[section]
\newtheorem{lemma}[theorem]{Lemma}
\newtheorem{corollary}[theorem]{Corollary}

\renewcommand{\P}{\mathbf{P}}
\renewcommand{\O}{\mathcal{O}}
\newcommand{\ti}{{\to \infty}}
\newcommand{\Prob}[1]{\mathbf{P}\left\{#1\right\}}
\newcommand{\E}{\mathbf{E}}
\newcommand{\Ind}{\mathbf{1}}
\newcommand{\standardspace}[1]{\mathbb{#1}}
\newcommand{\R}{\standardspace{R}}
\newcommand{\Sphere}{{\standardspace{S}^{d-1}}}
\renewcommand{\S}{{\standardspace{S}}}
\newcommand{\Ball}{{\standardspace{B}^d}}
\newcommand{\Beta}{\mathrm{B}}
\newcommand{\thf}{\frac{1}{2}}

\DeclareMathOperator{\peri}{peri}

\newcommand{\st}{\tilde{s}}

\begin{document}

\bibliographystyle{plain}

\title{$U$-max-Statistics}

\author{W.\ Lao\footnote{Institute of Stochastics,
    University of Karlsruhe, Englerstrasse 2, Karls\-ruhe, 76128
    Germany} \ and M.\ Mayer\footnote{Department of Mathematical
    Statistics and Actuarial Science, University of Bern,
    Sidlerstrasse 54, CH-3012 Bern, Switzerland. Supported by
    Swiss National Foundation Grant No.  200021-103579}}
\date{}
\bibliographystyle{plain}
\maketitle

\begin{abstract}
    In 1948, W. Hoeffding introduced a large class of unbiased
    estimators called $U$-statistics, defined as the average
    value of a real-valued $k$-variate function $h$ calculated at all
    possible sets of $k$ points from a random sample.
    In the present paper we investigate the corresponding
    extreme value analogue, which we shall
    call $U$-max-statistics. We are concerned with the behavior of the
    largest value of such function $h$ instead of its average. Examples
    of $U$-max-statistics are the diameter or the largest
    scalar product within a random sample. $U$-max-statistics of
    higher degrees are given by triameters and other metric
    invariants.


\end{abstract}

\section{Introduction}
$U$-statistics form a very important class of unbiased estimators
for distributional properties such as moments or Spearman's rank
correlation. A $U$-statistic of degree $k$ with symmetric kernel $h$
is a function of the form
\begin{displaymath}
    U(\xi_1, \dots, \xi_n) = {n \choose k}^{-1} \sum_J h(\xi_{i_1}, \cdots,
    \xi_{i_k}),
\end{displaymath}
where the sum stretches over $J = \{(i_1, \dots, i_k)\!: 1 \leq i_1
<\dots < i_k \leq n\}$, $\xi_1, \dots, \xi_n$ are random elements in
a measurable space $\mathcal{S}$ and $h$ is a real-valued Borel
function on $\mathcal{S}^k$, symmetric in its $k$ arguments. In his
seminal paper, Hoeffding \cite{hoeffding48} defined $U$-statistics
for not necessarily symmetric kernels and for random points in
$d$-dimensional Euclidean space $\R^d$. Later the concept was
extended to arbitrary measurable spaces. Since 1948, most of the
classical asymptotic results for sums of i.i.d.\ random variables
have been formulated in the setting of $U$-statistics, such as
central limit laws, strong laws of large numbers, Berry-Ess\'{e}en
type bounds and laws of the iterated logarithm.

The purpose of this article is to investigate the extreme
value analogue of $U$-statistics, i.e.
\begin{displaymath}
    H_n = \max_J h(\xi_{i_1}, \dots, \xi_{i_k}).
\end{displaymath}
A typical example of such $U$-max-statistic is the diameter of a
sample of points in a metric space, obtained by using the metric as
kernel. Grove and Markvorsen \cite{grove92} introduced an infinite
sequence of metric invariants generalizing the notion of diameter to
``triameter'', ``quadrameter'', etc. on compact metric spaces. Their
$k$-extent is the maximal average distance between $k$ points, which
is an example for a $U$-max-statistic of arbitrary degree $k$. Other
examples are the largest surface area or perimeter of a triangle
formed by point triplets, or the largest scalar product within a
sample of points in $\R^d$.

The key to our results is the observation that for all $z\in \R$,
the $U$-max-statistic $H_n$ does not exceed $z$ if and only if $U_z$
vanishes, where
\begin{displaymath}
    U_z = \sum_J\Ind\{h(\xi_{i_1}, \dots, \xi_{i_k}) > z\}.
\end{displaymath}
The random variable $U_z$ counts the number of exceedances of the
threshold $z$ and is a normalized $U$-statistic in the usual sense.
We approximate its distribution with the help of a Poisson
approximation result for the sum of dissociated random indicator
kernel functions by Barbour et al.\ \cite{bar:hol:jan92}, which
enables us to determine the distribution of $H_n$ up to some known
error. In order to deduce the corresponding limit law for $H_n$, the
behavior of the upper tail of the distribution of $h$ must be known,
which often requires complicated geometric computations. Denote by
$\| \cdot \|$ the Euclidean norm. The general results are used to
derive limit theorems for the following settings: largest interpoint
distance and scalar product of a sample of points in the
$d$-dimensional closed unit ball $\Ball = \{x \in \R^d\!: \| x \|
\leq 1\}$, where the directions of the points have a density on the
surface $\Sphere$ of $\Ball$ and are independent of the norms;
smallest spherical distance of a sample of points with density on
$\Sphere$; largest perimeter of all triangles formed by point
triplets in a sample of uniformly distributed points on the unit
circle $\S$.

\section{Poisson approximation for $U$-max-statistics}

The following result is easily derived from Theorem~2.N for
dissociated indicator random variables from Barbour et al.\
\cite{bar:hol:jan92}. We use the convention that improper sums for
$k=1$ equal zero.

\begin{theorem}\label{th:main}
    Let $\xi_1, \dots, \xi_n$ be i.i.d.\ $\mathcal{S}$-valued random elements
    and \\$h\!: \mathcal S^k \rightarrow \R$ a symmetric Borel
    function. Putting
    \begin{align*}
         p_{n,z} &= \Prob{h(\xi_1, \dots, \xi_k) > z}, \\
        \lambda_{n,z} &= {n \choose  k} p_{n,z}, \\
        \tau_{n,z}(r) &= p_{n,z}^{-1}\Prob{h(\xi_1, \dots, \xi_k) >
        z, h(\xi_{1+k-r}, \xi_{2+k-r}, \dots, \xi_{2k-r}) > z},
    \end{align*}
    we have, for any $n \geq k$ and any $z \in \R$,
    \begin{align}
        |&\Prob{H_n \leq z} - \exp\{-\lambda_{n,z}\}| \label{eq:main} \\
        &\leq (1-\exp\{-\lambda_{n,z}\})
        \left\{p_{n,z}\left[{n \choose k}-{n-k \choose k}\right] + \sum_{r =
        1}^{k-1}{k \choose r}{n-k \choose
        k-r}\tau_{n,z}(r)\right\}.\nonumber
    \end{align}
\end{theorem}

Clearly the result can be reformulated as well for the minimum value
of the kernel by replacing $h$ with $-h$. One of the main
applications of this theorem consists in determining a suitable
sequence of transformations \\$z_n\!: T \rightarrow \R$ with
$T\subset \R$, such that both the right hand side of (\ref{eq:main})
converges to zero as $n \ti$ for all $z=z_n(t)$, $t \in T$, and the
limits of $\exp\{-\lambda_{n,z_n(t)}\}$ are non-trivial for all $t
\in T$. The usual choice is $T =[0,\infty)$. One way to achieve this
goal is based on the following two remarks and will eventually lead
to the well known Poisson limit theorem of Silverman and Brown
\cite{sil:bro78}, originally proved by a suitable coupling.

\begin{remark}\label{re:order of tau}
    As already Silverman and Brown \cite{sil:bro78} stated,
    \begin{displaymath}
        p_{n,z} \leq \tau_{n,z}(1) \leq \dots \leq \tau_{n,z}(k) = 1.
    \end{displaymath}
\end{remark}

\begin{remark}\label{re:order of error}
    If the sample size $n$ tends to infinity, then the error
    (\ref{eq:main}) is asymptotically
    \begin{displaymath}
         \O\big(p_{n,z}n^{k-1}+\sum_{r =
         1}^{k-1}\tau_{n,z}(r) n^{k-r}\big)
    \end{displaymath}
    and for $k>1$ the sum is dominating, see \cite[p. 35]{bar:hol:jan92}.
\end{remark}

\begin{remark}
    The symmetry condition on $h$ can be avoided if $h$ is
    symmetrized by
    \begin{displaymath}
        h^*(x_1, \dots, x_k) = \max_{j_1, \dots, j_k}h(x_{j_1}, \dots, x_{j_k}),
    \end{displaymath}
    where the maximum is taken over all permutations of $1, \dots,
    k$.
\end{remark}

The conditions stated in \cite{sil:bro78} suffice to ensure that
Theorem~\ref{th:main} provides a non-trivial Weibull limit law.

\begin{corollary}[Silverman-Brown limit law \cite{sil:bro78}]~\label{co:sb}
    In the setting of Theorem~\ref{th:main}, if for some sequence of transformations
    $z_n\!: T \to \R$ with $T \subset \R$, the conditions
    \begin{equation}\label{eq:lambda}
         \lim_{n \ti}\lambda_{n,z_n(t)} = \lambda_t > 0
    \end{equation}
    and
    \begin{equation}\label{eq:correlation}
         \lim_{n\ti}n^{2k-1}p_{n,z_n(t)}\tau_{n,z_n(t)}(k-1) = 0
    \end{equation}
    hold for all $t \in T$, then
    \begin{equation}\label{eq:limit}
        \lim_{n\ti} \Prob{H_n \leq z_n(t)} = \exp\{-\lambda_t\}
    \end{equation}
    for all $t\in T$.
\end{corollary}

\begin{remark}\label{re:rate of convergence}
    Condition (\ref{eq:lambda}) implies $p_{n,z_n(t)}=\O(n^{-k})$
    and by Remarks~\ref{re:order of tau} and \ref{re:order of error} we obtain for (\ref{eq:limit})
    the rate of convergence
    \begin{displaymath}
        \O\big(n^{-1}+\sum_{r = 1}^{k-1}n^{2k-r}p_{n, z_n(t)}\tau_{n,
        z_n(t)}(r)\big)
    \end{displaymath}
    with upper bound
    \begin{equation}\label{eq:rate of convergence}
        \O(n^{2k-1}p_{n,z_n(t)}\tau_{n,z_n(t)}(k-1)).
    \end{equation}
    If $k > 2$, it is sometimes useful to replace (\ref{eq:correlation}) by
    the weaker requirement
    \begin{equation}\label{eq:better condition}
        \lim_{n\ti} n^{2k-r}p_{n,z_n(t)}\tau_{n,z_n(t)}(r)  = 0
    \end{equation}
    for each $r \in \{1, \dots, k-1\}$, a fact that follows
    immediately from Theorem~\ref{th:main} and Remark~\ref{re:order of
    error}.
\end{remark}

Appel and Russo \cite{app:rus06} obtained a Weibull limit law
similar to Corollary~\ref{co:sb} for bivariate $h$. They assume that
the upper tail of the distribution of $h(\xi_1, x)$ does not depend
on $x$ for almost all $x\in \mathcal{S}$, which implies that
(\ref{eq:lambda}) and (\ref{eq:correlation}) hold. However, this
condition is fulfilled only in very rare settings, e.g.\ for
uniformly distributed points on $\Sphere$.

\section{Largest interpoint distance} The asymptotic behavior of
the range of a univariate sample can be determined by classical
extreme value theory, see e.g.\ \cite[Sec.\ 2.9]{gal78}. The largest
interpoint distance
\begin{displaymath}
    H_n = \max_{1 \leq i < j \leq n}\|\xi_i - \xi_j\|
\end{displaymath}
within a sample of points in $\R^d$ is a natural and consistent
generalization of the range to spatial data. Matthews and Rukhin
\cite{mat:ruk93} derived its limiting behavior for a normal sample,
a work which has been generalized by Henze and Klein
\cite{hen:klein96} to a sample of points with symmetric Kotz
distribution. Appel et al.\ \cite{app:naj:rus02} found corresponding
limit laws in the setting of uniformly distributed points in
2-dimensional compact sets, which are not too smooth near the
endpoints of their largest axes. They also provided bounds for the
limit law of the diameter of uniformly distributed points in
ellipses and the unit disk. The exact limit distribution for the
disk and in more general settings was found independently by Lao
\cite{lao06} and Mayer and Molchanov \cite{mayer06}. Lao
\cite{lao06} used Theorem~A of \cite{sil:bro78} to obtain the exact
limit law for the diameter of a uniform sample in $\Ball$. The
results in \cite{mayer06} rely on a combination of geometric
considerations and blocking techniques and yield e.g.\ the special
case of Theorem~\ref{th:metric} for spherically symmetric
distributions.

In what follows, we denote by $\langle \cdot, \cdot \rangle$ the
scalar product, by $\mu_{d-1}$ the $(d-1)$-dimensional Hausdorff
measure and by $\Gamma$ and $\Beta$ the complete Gamma and Beta
functions.

\begin{theorem}
  \label{th:metric}
  Let $\xi_1, \xi_2, \dots$ be i.i.d.\ points in $\Ball$, $d\geq 2$, such
  that\\
  $\xi_i=\|\xi_i\| U_i$, $i \geq 1$, where $U_i$ and $\|\xi_i\|$ are
  independent and $U_i \in \Sphere$. Assume that the
  distribution function $F$ of $1-\|\xi_1\|$ satisfies
  \begin{displaymath}
        \lim_{s \downarrow 0} s^{-\alpha}F(s)=a \in (0, \infty)
  \end{displaymath}
  for some $\alpha \geq 0$. Further assume that $U_1$ has a density $f$
  with respect to $\mu_{d-1}$ and that $\int_\Sphere f(x)f(-x) \mu_{d-1}(dx) \in (0, \infty)$.
  Then
  \begin{displaymath}
    \lim_{n\ti} \Prob{n^{2/\gamma}(2-H_n) \leq t}
    =1-\exp\left\{-\frac{\sigma_1}{2}t^\gamma\right\}
  \end{displaymath}
  for $t>0$, where
  \begin{displaymath}
        \gamma=(d-1)/2+2\alpha
  \end{displaymath}
  and
  \begin{displaymath}
        \sigma_1=\frac{(4\pi)^\frac{d-1}{2}a^2\Gamma^2(\alpha+1)}{\Gamma(\frac{d+1}{2}+2\alpha)}\int_\Sphere
        f(x)f(-x)\mu_{d-1}(dx).
  \end{displaymath}
  The rate of convergence for $t<\infty$ is $\O(n^{-\frac{d-1}{d-1+4\alpha}})$.
\end{theorem}

\begin{remark}
    Spherically symmetric distributed points have independent
    and uniformly distributed directions and hence \cite[Th.\ 4.2]{mayer06}
    follows immediately from Theorem~\ref{th:metric}
    with
    \begin{displaymath}
        \int_\Sphere f(x)f(-x)\mu_{d-1}(dx) = \frac{\Gamma(\frac{d}{2})}{2\pi^{d/2}}.
    \end{displaymath}
    The special case $\alpha = 1$ and $a=d$ yields the limit law for
    the diameter of a sample of uniformly distributed points in
    $\Ball$, see \cite{lao06} or \cite{mayer06}.
\end{remark}

\begin{remark}
    If $\|\xi_i\| = 1$ almost surely, then $\alpha=0$ and
    $a=1$. For instance, if $U_i$ are uniformly distributed on $\Sphere$,
    then for $t>0$
    \begin{displaymath}
        \lim_{n\ti}\Prob{n^{4/(d-1)}(2-H_n) \leq t}
        = 1-\exp\left\{-\frac{2^{d-3}\Gamma(\frac{d}{2})}{\pi^\thf\Gamma(\frac{d+1}{2})}t^\frac{d-1}{2}\right\},
    \end{displaymath}
    see \cite{app:rus06} or \cite{mayer06}.
    Another example appears if $U_i$ has the von Mises-Fisher distribution of
    dimension $d\geq 2$ with density
    \begin{displaymath}
            f_F(x) = C_d(\kappa)\exp\left\{\kappa \langle \mu, x \rangle
            \right\}
    \end{displaymath}
    for $x \in \Sphere$, where $\mu \in \Sphere$ represents the mean
    direction and $\kappa > 0$ is the concentration parameter.
    The normalizing constant $C_d(\kappa)$ is given by
    \begin{displaymath}
            C_d(\kappa) = \frac{\kappa^{d/2-1}}{(2\pi)^{d/2}I_{d/2-1}(\kappa)},
    \end{displaymath}
    where $I_\nu$ denotes the modified Bessel function of the first kind
    of order $\nu$.
    With
    \begin{displaymath}
        \int_\Sphere f_F(x)f_F(-x)\mu_{d-1}(dx) =
        C_d^2(\kappa)\frac{2\pi^{d/2}}{\Gamma(\frac{d}{2})}
    \end{displaymath}
    the corresponding limit law follows immediately.
\end{remark}

A key part of the proof of Theorem~\ref{th:metric} is the asymptotic
tail behavior of the distribution of the distance between two
i.i.d.\ points.

\begin{lemma}\label{le:metric}
    If the conditions of Theorem~\ref{th:metric} hold,
    then
    \begin{displaymath}
        \lim_{s\downarrow 0}s^{-\gamma}\Prob{\|\xi_1-\xi_2\| \geq
        2-s} = \sigma_1.
    \end{displaymath}
\end{lemma}

\begin{proof}
    Let $\eta_1$ and $\eta_2$ be independent random variables with distribution
    $F$ and denote by $\beta_x$ the smaller central angle between
    $U_2$ and $x\in \Sphere$. The cosine theorem yields
    \begin{align*}
        \Prob{\|\xi_1-\xi_2\| \geq 2-s} &= \Prob{\|\xi_1\|^2+\|\xi_2\|^2 + 2\|\xi_1\|\|\xi_2\|\cos
        \beta_{-U_1} \geq(2-s)^2} \\
        &=\Prob{\cos \beta_{-U_1} \geq
        \frac{(2-s)^2-(1-\eta_1)^2-(1-\eta_2)^2}{2(1-\eta_1)(1-\eta_2)}}
    \end{align*}
    and by expansion of $\cos\beta_{-U_1}$ about 0 we obtain for sufficiently small $s$
    \begin{equation}\label{eq:metric expansion}
        \Prob{\|\xi_1-\xi_2\| \geq 2-s} = \Prob{|\beta_{-U_1} | \leq
        2(\st-\eta_1-\eta_2)^\thf, \eta_1+\eta_2 \leq \st},
    \end{equation}
    where $|\st - s| \leq C_1 s^2$ for some finite $C_1$, thus
    $\st/s \to 1$ as $s \downarrow 0$. Lebesgue's differentiation
    theorem (see e.g.\ \cite[Th.\ 2.9.5]{federer69}) implies that
    \begin{equation}\label{eq:metric continuity}
        \lim_{s\downarrow 0}\frac{\Prob{|\beta_{-x}|\leq
        2(\st-y)^\thf}}{(4(\st-y))^{\frac{d-1}{2}}}
        = \mu_{d-1}(\mathbb{B}^{d-1})f(-x)=
        \frac{\pi^\frac{d-1}{2}}{\Gamma(\frac{d+1}{2})}f(-x)
    \end{equation}
    for $\mu_{d-1}$-almost every $x\in \Sphere$ and any $y \in [0, \st]$.
    Integration over all $x \in \Sphere$ with respect to $f$ yields
    \begin{displaymath}
        \lim_{s\downarrow 0} (\st-y)^{-\frac{d-1}{2}} \Prob{|\beta_{-U_1}|\leq
        2(\st-y)^\thf} = c,
    \end{displaymath}
    where
    \begin{displaymath}
        c=\frac{(4\pi)^\frac{d-1}{2}}{\Gamma(\frac{d+1}{2})}\int_\Sphere
        f(x)f(-x)\mu_{d-1}(dx),
    \end{displaymath}
    and hence with (\ref{eq:metric expansion})
    \begin{displaymath}
        \lim_{s\downarrow 0}\frac{\Prob{\|\xi_1-\xi_2\| \geq 2-s}}
        {\E\left((\st-\eta_1-\eta_2)^\frac{d-1}{2}\Ind\{\eta_1+\eta_2\leq \st\}\right)} = c.
    \end{displaymath}
    If $\alpha=0$, then $\Prob{\eta_i=0}=a$, $i=1,2$, and thus
    \begin{displaymath}
        \lim_{s \downarrow 0} \st^{-\gamma} \Prob{\|\xi_1-\xi_2\| \geq
        2-s} = c a^2 = \sigma_1.
    \end{displaymath}
    If $\alpha > 0$,
    \begin{displaymath}
        \lim_{s\downarrow 0}\frac{\Prob{\|\xi_1-\xi_2\| \geq 2-s}}
        {\int_0^{\st}\int_0^{\st-y_1} (\st - y_1 - y_2)^\frac{d-1}{2} d F(y_2) d
        F(y_1)}=c
    \end{displaymath}
    and substituting $v_i = y_i/\st$, $i = 1, 2$, yields
    \begin{displaymath}
        \lim_{s\downarrow 0}\frac{\Prob{\|\xi_1-\xi_2\|
        \geq2-s}}{\st^\gamma}
        =ca^2\alpha^2\int_0^1\int_0^{1-v_1}
        (1-v_1-v_2)^\frac{d-1}{2}(v_1v_2)^{\alpha-1}dv_2dv_1.
    \end{displaymath}
    By Dirichlet's Formula, the double integral equals
    \begin{displaymath}
        \frac{\Gamma^2(\alpha)\Gamma(\frac{d+1}{2})}{\Gamma(\frac{d+1}{2}
        +2\alpha)}
    \end{displaymath}
    and the proof is complete.
\end{proof}

\begin{proof}[Proof of Theorem~\ref{th:metric}]
    Plugging the transformation $z_n(t) = 2-tn^{-2/\gamma}$, $t > 0$
    into Corollary~\ref{co:sb} and using the tail probabilities given
    in Lemma~\ref{le:metric}, we find
    \begin{displaymath}
        \lim_{n \ti}{n \choose 2}\Prob{\|\xi_1 - \xi_2\| > z_n(t)} =
        \frac{\sigma_1}{2}t^\gamma, \qquad t > 0.
    \end{displaymath}
    Hence condition (\ref{eq:lambda}) holds for all $t>0$.
    The more extensive part of the proof aims to show that
    (\ref{eq:correlation}) holds. Let $\beta_x$ and $\beta'_x$ be the smaller central
    angles between $U_2$ and $x\in \Sphere$ and between $U_3$ and $x\in \Sphere$. Further let
    $\eta_1$, $\eta_2$ and $\eta_3$ be independent random variables with distribution
    $F$. Put $s_n=tn^{-2/\gamma}$. Following the proof of Lemma~\ref{le:metric}
    \begin{align}
        &\Prob{\|\xi_1-\xi_2\| > z_n(t), \|\xi_1-\xi_3\| > z_n(t)} \nonumber\\
        &\leq \Prob{|\beta_{-U_1}| \leq 2s_n^\thf, |\beta'_{-U_1}| \leq 2 s_n^\thf, \eta_i \leq
            s_n, i = 1, 2, 3}\nonumber\\
        &= \E\left(\int_\Sphere\P\big\{|\beta_{-x}|\leq 2 s_n^\thf\big\}^2 f(x) \mu_{d-1}(dx)\Ind\{\eta_i \leq
        s_n, i = 1, 2, 3\}\right)\nonumber\\
        & \leq C\E(s_n^{d-1}\Ind\{\eta_i \leq
        s_n, i = 1, 2, 3\}),\label{eq:metric bounds}
    \end{align}
    where the last step follows from (\ref{eq:metric continuity})
    and $C$ is a suitable finite positive constant.
    If $\alpha = 0$, then $\Prob{\eta_i=0}=a$, $i = 1, 2, 3$, and we
    obtain
    \begin{align*}
        \lim_{n\ti} n^3 &\Prob{\|\xi_1-\xi_2\| > z_n(t), \|\xi_1-\xi_3\| >
        z_n(t)} \\
        &\leq C a^3\lim_{n\ti}n^3 s_n^{d-1} = C a^3 t^{d-1}\lim_{n\ti}n^{-1}=0.
    \end{align*}
    If $\alpha>0$, we derive from (\ref{eq:metric bounds}) that
    \begin{displaymath}
        \lim_{n\ti} \frac{\Prob{\|\xi_1-\xi_2\| > z_n(t), \|\xi_1-\xi_3\| >
        z_n(t)}}{C \int_0^{s_n} \int_0^{s_n} \int_0^{s_n} s_n^{d-1} d F(y_3) d
        F(y_2) d F(y_1)} \leq 1
    \end{displaymath}
    and substituting $v_i = y_i/s_n$, $i = 1, 2, 3$, yields
    \begin{align*}
        \lim_{n\ti} n^3 &\Prob{\|\xi_1-\xi_2\| > z_n(t), \|\xi_1-\xi_3\| > z_n(t)} \\
        &\leq C a^3 \lim_{n\ti}s_n^{d-1+3\alpha} =
        C a^3 t^{d-1+3\alpha}\lim_{n\ti}n^{-\frac{d-1}{d-1+4\alpha}} =
        0.
    \end{align*}
    The rate of convergence is determined via (\ref{eq:rate of convergence}).
\end{proof}

\section{Largest scalar product} Besides the Euclidean metric, the
scalar product is another symmetric kernel on $\R^d\times \R^d$. The
behavior of its largest value
\begin{displaymath}
    H_n = \max_{1 \leq i < j \leq n}\langle \xi_i, \xi_j\rangle
\end{displaymath}
within a sample of points in $\Ball$ is determined in the next
result.

\begin{theorem}\label{th:scalar product}
  Let $\xi_1, \xi_2, \dots$ be i.i.d.\ points in $\Ball$, $d\geq 2$, such
  that\\
  $\xi_i=\|\xi_i\| U_i$, $i \geq 1$, where $U_i$ and $\|\xi_i\|$ are
  independent and $U_i \in \Sphere$. Assume that the
  distribution function $F$ of $1-\|\xi_1\|$ satisfies
  \begin{displaymath}
        \lim_{s \downarrow 0} s^{-\alpha}F(s)=a \in (0, \infty)
  \end{displaymath}
  for some $\alpha \geq 0$. Further assume that $U_1$ has a
  square-integrable density $f$ on $\Sphere$ with respect to $\mu_{d-1}$. Then
  \begin{displaymath}
    \lim_{n\ti} \Prob{n^{2/\gamma}(1-H_n) \leq t}
    =1-\exp\left\{-\frac{\sigma_2}{2}t^\gamma\right\}
  \end{displaymath}
  for $t>0$, where
  \begin{displaymath}
        \gamma=(d-1)/2+2\alpha
  \end{displaymath}
  and
  \begin{displaymath}
        \sigma_2=\frac{(2\pi)^\frac{d-1}{2}a^2\Gamma^2(\alpha+1)}{\Gamma(\frac{d+1}{2}+2\alpha)}\int_\Sphere
        f^2(x)\mu_{d-1}(dx).
  \end{displaymath}
  The rate of convergence for $t<\infty$ is $\O(n^{-\frac{d-1}{d-1+4\alpha}})$.
\end{theorem}

\begin{lemma}\label{le:scalar product}
    If the conditions of Theorem~\ref{th:scalar product} hold,
    then
    \begin{displaymath}
        \lim_{s\downarrow 0}s^{-\gamma}\Prob{\langle \xi_1, \xi_2 \rangle \geq
        1-s} = \sigma_2.
    \end{displaymath}
\end{lemma}

\begin{proof}
    If $\beta_x$ is the smaller central angle between $U_2$ and $x\in
    \Sphere$ and $\eta$ is distributed as $1 - \|\xi_1\|\|\xi_2\|$, then
    \begin{displaymath}
        \Prob{\langle \xi_1, \xi_2 \rangle \geq 1-s} = \Prob{\cos
        \beta_{U_1} \geq (1-s)/(1-\eta), \eta \leq s}.
    \end{displaymath}
    Expanding $\cos \beta_{U_1}$ about 0 yields for all sufficiently small $s$
    \begin{equation}\label{eq:scalar product expansion}
         \Prob{\langle \xi_1, \xi_2 \rangle \geq 1-s} = \Prob{|\beta_{U_1}| \leq (2(\st-\eta))^\thf, \eta \leq
         \st},
    \end{equation}
    where $|\st - s| \leq C_1 s^2$ for some finite $C_1$, and thus
    $\st/s \to 1$ as $s\downarrow 0$. Lebesgue's differentiation theorem (see e.g.\ \cite[Th.\ 2.9.5]{federer69}) implies that
    \begin{equation}\label{eq:scalar product continuity}
        \lim_{s\downarrow 0}
        \frac{\Prob{|\beta_x|\leq(2(\st-y))^\thf}}{(2(\st-y))^\frac{d-1}{2}}
        = \mu_{d-1}(\mathbb{B}^{d-1})f(x) =
        \frac{\pi^\frac{d-1}{2}}{\Gamma(\frac{d+1}{2})}f(x).
    \end{equation}
    for $\mu_{d-1}$-almost every $x\in \Sphere$ and any $y\in [0,
    \st]$. Integration over all $x \in \Sphere$ with respect to $f$ yields
    \begin{displaymath}
        \lim_{s\downarrow 0} (\st-y)^{-\frac{d-1}{2}} \Prob{|\beta_{U_1}|\leq
        (2(\st-y))^\thf} = c
    \end{displaymath}
    with
    \begin{displaymath}
        c = \frac{(2\pi)^\frac{d-1}{2}}{\Gamma(\frac{d+1}{2})}\int_\Sphere
        f^2(x)\mu_{d-1}(dx),
    \end{displaymath}
    and by (\ref{eq:scalar product expansion}) we obtain
    \begin{displaymath}
        \lim_{s \downarrow 0}
        \Prob{\langle \xi_1, \xi_2 \rangle \geq 1-s}\big/\E\left((\st-\eta)^\frac{d-1}{2}\Ind\{\eta\leq \st\}\right) = c.
    \end{displaymath}
    If $\alpha=0$, then $\Prob{\eta=0}=a^2$ and hence
    \begin{displaymath}
        \lim_{s\downarrow 0}s^{-\gamma}\Prob{\langle \xi_1, \xi_2 \rangle \geq
        1-s} = ca^2=\sigma_2.
    \end{displaymath}
    If $\alpha > 0$, then $\Prob{\langle \xi_1, \xi_2 \rangle \geq
        1-s}$ equals asymptotically, as $s \downarrow 0$, to
    \begin{displaymath}
        c \int_{1-\st}^1\int_{(1-\st)/y_1}^1(\st-1+y_1y_2)^\frac{d-1}{2}d F(1-y_2) d F(1-y_1).
    \end{displaymath}
    By substituting $v_1=(1-y_1)/\st$ and $v_2=(1-y_2)/(1-(1-\st)/y_1))$
    the last expression equals asymptotically, as $s \downarrow 0$, to
    \begin{align*}
       ca^2\alpha^2\int_0^1 \int_0^1 & \st^{2\alpha}\left(\frac{1-v_1}{1-\st v_1}\right)^\alpha\left(\st-1+(1-\st v_1)(1-\st v_2\frac{1-v_1}{1-\st v_1})\right)^\frac{d-1}{2}\\
       &v_1^{\alpha-1}v_2^{\alpha-1}d v_2 d v_1.
    \end{align*}
    Hence
    \begin{align*}
        \lim_{s \downarrow 0} & s^{-\gamma}\Prob{\langle \xi_1, \xi_2 \rangle \geq
            1-s}\\
            &= ca^2\alpha^2\int_0^1(1-v_1)^{\frac{d-1}{2}+\alpha}v_1^{\alpha-1}d
            v_1 \int_0^1 (1-v_2)^\frac{d-1}{2}v_2^{\alpha-1} d v_2\\
            &= ca^2\alpha^2\Beta((d+1)/2+\alpha,
            \alpha)\Beta((d+1)/2, \alpha) = \sigma_2.
    \end{align*}
\end{proof}

\begin{proof}[Proof of Theorem~\ref{th:scalar product}]
    An application of Corollary~\ref{co:sb} yields, together with
    the transformation $z_n=1-tn^{-2/\gamma}$, $t>0$, and Lemma~\ref{le:scalar product} the limit
    \begin{displaymath}
        \lim_{n\ti} {n \choose 2}\Prob{\langle \xi_1, \xi_2\rangle \geq
        z_n(t)} = \frac{\sigma_2}{2}t^\gamma
    \end{displaymath}
    hence (\ref{eq:lambda}) holds for any $t >0$ and it remains to
    check (\ref{eq:correlation}). Put $s_n = tn^{-2/\gamma}$ and let $\beta_x$ and $\beta'_x$ be the smaller central
    angles between $U_2$ and $x\in \Sphere$ and between $U_3$ and
    $x$. Following the proof of Lemma~\ref{le:scalar product}
    \begin{align}
        &\Prob{\langle \xi_1, \xi_2 \rangle \geq z_n(t), \langle \xi_1, \xi_3 \rangle \geq
        z_n(t)} \nonumber\\
        &\leq \Prob{|\beta_{U_1}|\leq (2s_n)^\thf, |\beta'_{U_1}| \leq (2s_n)^\thf, \|\xi_i\| \geq z_n(t), i = 1, 2,
        3}\nonumber\\
        &= \E\left(\int_\Sphere \Prob{|\beta_x|\leq(2s_n)^\thf}^2 f(x) \mu_{d-1}(dx)\Ind\{\|\xi_i\| \geq z_n(t), i = 1, 2,
        3\}\right)\nonumber\\
        &\leq C \E(s_n^{d-1}\Ind\{\|\xi_i\| \geq z_n(t), i = 1, 2,
        3\}), \label{eq:scalar product expectation}
    \end{align}
    where the last step follows from (\ref{eq:scalar product
    continuity}) and $C$ is a suitable finite positive constant.
    If $\alpha = 0$, then $\Prob{\|\xi_i\|=1} = \alpha$, $i = 1,
    2, 3$, and hence
    \begin{align*}
        n^3 &\Prob{\langle \xi_1, \xi_2 \rangle \geq z_n(t), \langle \xi_1, \xi_3 \rangle \geq
        z_n(t)}\\
        &\leq C a^3\lim_{n\ti}n^3s_n^{d-1}= C a^3 t^{d-1}\lim_{n\ti}n^{-1} =0.
    \end{align*}
    If $\alpha > 0$, then (\ref{eq:scalar product expectation}) is
    bounded from above by
    \begin{displaymath}
        C \int_0^{s_n} \int_0^{s_n} \int_0^{s_n}s_n^{d-1}d F(y_3)
        d F(y_2) d F(y_1)
    \end{displaymath}
    and substituting $v_i=y_i/s_n$, $i = 1, 2, 3$, yields finally
    \begin{align*}
        \lim_{n\ti} n^3 &\Prob{\langle \xi_1, \xi_2 \rangle \geq z_n(t), \langle \xi_1, \xi_3 \rangle \geq
        z_n(t)} \\
        &\leq C a^3 \lim_{n\ti}s_n^{d-1+3\alpha} =
        C a^3 t^{d-1+3\alpha}\lim_{n\ti}n^\frac{d-1}{d-1+4\alpha} = 0,
    \end{align*}
    and the rate of convergence is determined via (\ref{eq:rate of
    convergence}).
\end{proof}

\section{Smallest spherical distance}

A nice application of Theorem~\ref{th:scalar product} comes from the
field of directional statistics. The following theorem determines
the limiting behavior of the smallest spherical distance
\begin{displaymath}
    S_n = \min_{1 \leq i < j \leq n} \beta_{i, j}
\end{displaymath}
within i.i.d.\ points $U_1, U_2, \dots$ on $\Sphere$, where
$\beta_{i,j}$ denotes the smaller of the two central angles between
$U_i$ and $U_j$. In other words, $S_n$ equals the smallest central
angle formed by point pairs within the sample.

\begin{theorem}\label{th:spherical distance}
    Let $U_1, U_2 \dots$ be i.i.d.\ points on $\Sphere$, $d\geq 2$,
    having square-integrable density $f$ on $\Sphere$ with respect to $\mu_{d-1}$. Then
    \begin{displaymath}
        \lim_{n\ti}\Prob{n^{2/(d-1)}S_n \leq t} =
        1-\exp\left\{-\frac{\sigma_3}{2}t^{d-1}\right\}
    \end{displaymath}
    for any $t > 0$, where
    \begin{displaymath}
        \sigma_3 = \frac{\pi^\frac{d-1}{2}}{\Gamma(\frac{d+1}{2})}\int_\Sphere f^2(x)\mu_{d-1}(dx)
    \end{displaymath}
    The rate of convergence is $\O(n^{-\thf})$ for finite $t$.
\end{theorem}

If the points are uniformly distributed on $\Sphere$, then
Theorem~\ref{th:spherical distance} applies with
\begin{displaymath}
    \int_\Sphere f^2(x)\mu_{d-1}(dx) =
    \frac{\Gamma(\frac{d}{2})}{2\pi^{d/2}}.
\end{displaymath}
If the points on $\Sphere$ follow the von Mises-Fisher distribution
as introduced in Section 3, then
\begin{displaymath}
    \int_\Sphere f^2_F(x)\mu_{d-1}(dx) =
    C_d^2(\kappa)/C_d(2\kappa).
\end{displaymath}

In dimension 2, $S_n$ equals the minimal spacing, i.e.\ the smallest
arc length between the ``order'' statistics.

\begin{proof}[Proof of Theorem \ref{th:spherical distance}]
    Clearly, the relation $\cos \beta_{i,j} = \langle U_i, U_j
    \rangle$ holds for all pairs of $i$ and $j$ between 1 and $n$.
    Since the cosine function is continuous and monotone strictly
    decreasing on $(0,\pi)$ and by the fact that
    \begin{displaymath}
        \lim_{s \downarrow 0} s^{-\thf}\arccos(1-s)=\sqrt{2},
    \end{displaymath}
    it follows that
    \begin{align*}
        & \lim_{n\ti}\Prob{n^{2/(d-1)}S_n \leq t} = \lim_{n\ti}\Prob{\min_{1\leq i < j \leq n}\beta_{i,j} \leq
       tn^{-2/(d-1)}}\\
        =& \lim_{n\ti}\Prob{\min_{1\leq i < j \leq n}\beta_{i,j}
        \leq \arccos\big(1-t^2n^{-4/(d-1)}/2\big)} \\
        =&\lim_{n \ti}\Prob{\max_{1 \leq i < j \leq n} \langle U_i, U_j
        \rangle \geq 1-t^2 n^{-4/(d-1)}/2}.
    \end{align*}
    Theorem~\ref{th:scalar product} yields the proof with $\alpha=0$ and $a=1$.
\end{proof}

\section{Largest perimeter}
Finally we present a result for a $U$-max-statistic of degree 3,
namely the limit law for the largest value
\begin{displaymath}
    \max_{1 \leq i < j < \ell \leq n} \peri(U_i, U_j, U_\ell)
\end{displaymath}
of the perimeter $\peri(U_i, U_j, U_\ell)$ of all triangles formed
by triplets of independent and uniformly distributed points $U_1,
U_2, \dots$ on the unit circle $\S$. The random triameter (see
\cite{grove92}) of the sample is the largest perimeter up to a
factor 3, hence the limit law for the triameter of $U_1, U_2, \dots$
can be derived immediately.

\begin{theorem}\label{th:perimeter}
    If $U_1, U_2, \dots$ are independent and uniformly distributed points on $\S$, then
    \begin{displaymath}
        \lim_{n\ti} \Prob{n^3(3\sqrt{3}-H_n) \leq t} = 1-\exp\left\{-\frac{2t}{9\pi}\right\}
    \end{displaymath}
    for all $t>0$ and for finite $t$ the rate of convergence is $\O(n^{-\thf})$.
\end{theorem}

\begin{lemma}\label{le:perimeter}
    If $U_1, U_2, U_3$ are independent and uniformly distributed
    points on $\S$, then
    \begin{displaymath}
        \lim_{s\downarrow 0}s^{-1}\Prob{\peri(U_1, U_2, U_3) \geq
        3\sqrt{3}-s} = \frac{4}{3\pi}.
    \end{displaymath}
\end{lemma}

\begin{proof}
    Clearly, $\peri(x_1, x_2, x_3)$ is maximal for $x_1, x_2, x_3$
    being the vertices of an equilateral triangle on $\S$, which has
    perimeter $3\sqrt{3}$. If $\beta_1$ and $\beta_2$ are the
    angles (measured counter-clockwise) between $U_1$ and $U_2$
    and between $U_2$ and $U_3$ respectively. By rotational symmetry,
    $\beta_1$ and $\beta_2$ are independent and uniformly distributed on $[0,2\pi]$.
    The cosine theorem yields for sufficiently small $s$
    \begin{align}
        &\Prob{\peri(U_1, U_2, U_3) \geq
        3\sqrt{3}-s} = 2\mathbf{P}\big\{(2-2\cos\beta_1)^\thf +
        (2-2\cos\beta_2)^\thf \nonumber\\
            &\qquad + (2-2\cos(2\pi - \beta_1-\beta_2))^\thf \geq
            3\sqrt{3}-s, \beta_1, \beta_2 \in [2\pi/3\pm c_s]\big\},\label{eq:peri}
    \end{align}
    where $c_s=C_1 \sqrt{s}$ and $C_1$ is a suitable finite positive constant.
    If $\eta_1$ and $\eta_2$ are independent and uniformly distributed
    on $[-c_s,c_s]$, then the last expression
    equals
    \begin{align*}
        2&\mathbf{P}\big\{(2-2\cos(2\pi/3+\eta_1))^\thf +
        (2-2\cos(2\pi/3+\eta_2))^\thf\\
        & \qquad + (2-2\cos(2\pi/3-\eta_1-\eta_2))^\thf \geq
            3\sqrt{3}-s\big\} \Prob{\beta_1 \in [2\pi/3\pm c_s]}^2.
    \end{align*}
    By series expansion, (\ref{eq:peri}) equals
    \begin{align}
        & 2(c_s/\pi)^2 \Prob{\eta_1^2+\eta_2^2+(\eta_1 + \eta_2)^2\leq 8 \st /\sqrt{3}} \label{eq:peri expansion}\\
        &=2(c_s/\pi)^2 \Prob{\eta_2\in \left[-\eta_1/2\pm (4\st/\sqrt{3}-3\eta_1^2/4)^\thf \right]} \nonumber\\
        &= \pi^{-2}\int_{-4\sqrt{\st}/3^{3/4}}^{4\sqrt{\st}/3^{3/4}}(4\st/\sqrt{3}-3y^2/4)^\thf dy =
        \frac{4\st}{3\pi}, \nonumber
    \end{align}
    where $|\st - s| \leq C_2 s^{3/2}$ for some finite $C_2$, and the proof follows by the fact
    that $\st/s \to 1$ as $s \downarrow 0$.
\end{proof}

\begin{proof}[Proof of Theorem~\ref{th:perimeter}]
    Plugging into Corollary~\ref{co:sb} the transformation \\$z_n(t) =
    3\sqrt{3}-tn^{-3}$ together with the result of Lemma~\ref{le:perimeter} yields
    \begin{displaymath}
        \lim_{n\ti}{n \choose 3}\Prob{\peri(U_1, U_2, U_3) >
        z_n(t)} = \frac{2t}{9\pi},
    \end{displaymath}
    hence (\ref{eq:lambda}) is satisfied for all $t>0$. Condition (\ref{eq:correlation})
    does not hold, so we use the weaker requirement (\ref{eq:better condition}) to replace
    (\ref{eq:correlation}), i.e.\ we need to show that
    \begin{equation}
        \lim_{n\ti}n^5\Prob{\peri(U_1, U_2, U_3) > z_n(t), \peri(U_1, U_4, U_5)
        > z_n(t)} =0\label{eq:peri cond 1}
    \end{equation}
    and
    \begin{equation}
        \lim_{n\ti}n^4\Prob{\peri(U_1, U_2, U_3) > z_n(t),
        \peri(U_1, U_2, U_4) > z_n(t)} =0\label{eq:peri cond 2}.
    \end{equation}
    For (\ref{eq:peri cond 1}), we follow the
    proof of Lemma~\ref{le:perimeter}. In addition, denote by
    $\beta_1'$ and $\beta_2'$ the random angles between $U_1$ and
    $U_4$ and between $U_4$ and $U_5$ respectively. It follows
    immediately by rotational symmetry, that $\beta_1, \beta_2,
    \beta_1'$ and $\beta_2'$ are independent and uniformly distributed on $[0,2\pi]$.
    With the help of Lemma~\ref{le:perimeter} we check (\ref{eq:peri cond 1})
    by
    \begin{align*}
        &\lim_{n\ti}n^5\Prob{\peri(U_1, U_2, U_3) > z_n(t),
        \peri(U_1, U_2, U_4) > z_n(t)} \\
        &\leq C_1 \lim_{n\ti}n^5\Prob{\peri(U_1, U_2, U_3) >
        z_n(t)}^2 = C_2 t^2 \lim_{n\ti} n^{-1}=0,
    \end{align*}
    where $C_1$ and $C_2$ are suitable finite positive constants.
    To show (\ref{eq:peri cond 2}) we follow  the proof of
    Lemma~\ref{le:perimeter} and introduce the random variable
    $\eta_3$, which is independent of $\eta_1$ and $\eta_2$ and uniformly
    distributed on $[-c_s,c_s]$. For suitable finite positive constants $C_3$, $C_4$ and $C_5$
    \begin{align*}
        &\Prob{\peri(U_1, U_2, U_3) > z_n(t),
        \peri(U_1, U_2, U_4) > z_n(t)} \\
        &\leq C_3 c_s^3\Prob{\eta_2, \eta_3 \in \left[-\eta_1/2\pm
        (4\st/\sqrt{3}-3\eta_1^2/4)^\thf \right]} \\
        &=C_4c_s^2\int_{-4\sqrt{\st}/3^{3/4}}^{4\sqrt{\st}/3^{3/4}}\Prob{\eta_2 \in \left[-y/2\pm
        (4\st/\sqrt{3}-3y^2/4)^\thf \right]}^2dy = C_5\st^{3/2},
    \end{align*}
    and with $s=tn^{-3}$ and $s / \st \to 1$ as $ s\to 0$
    \begin{align*}
        &\lim_{n\ti}n^4\Prob{\peri(U_1, U_2, U_3) > z_n(t),
        \peri(U_1, U_2, U_4) > z_n(t)} \\
        &\leq C_5 t^{3/2} \lim_{n\ti}n^{-\thf}=0.
    \end{align*}
    Hence (\ref{eq:peri cond 2}) holds and the rate of convergence is determined by Remark~\ref{re:order of
    error}.
\end{proof}

\section*{Acknowledgements}
The authors would like to thank Prof.\ Dr.\ N.\ Henze and Prof.\
Dr.\ I.\ Molchanov for their invaluable help concerning this and
many other problems.

\end{document}